\documentclass[11pt, reqno]{amsart}

\usepackage{color,graphics,epsfig}

\newcommand{\calA}{\mathcal{A}}

\newcommand{\calM}{\mathcal{M}}

\newcommand{\calP}{\mathcal{P}}

\newcommand{\calW}{\mathcal{W}}

\newcommand{\mC}{\mathbb{C}}
\newcommand{\mD}{\mathbb{D}}

\newcommand{\mF}{\mathbb{F}}

\newcommand{\mN}{\mathbb{N}}

\newcommand{\mR}{\mathbb{R}}
\newcommand{\mS}{\mathbb{S}}
\newcommand{\mT}{\mathbb{T}}

\newcommand{\mZ}{\mathbb{Z}}

\newcommand{\inv}{{\textrm{inv }}}
\newcommand{\ind}{{\textrm{ind }}}

\newtheorem{theorem}{Theorem}[section]
\newtheorem{lemma}[theorem]{Lemma}
\newtheorem{corollary}[theorem]{Corollary}
\newtheorem{proposition}[theorem]{Proposition}

\theoremstyle{definition}
\newtheorem{remark}[theorem]{Remark}

\theoremstyle{definition}
\newtheorem{definition}[theorem]{Definition}

\theoremstyle{definition}

\begin{document}

\keywords{Nyquist criterion, control theory, Banach algebras}

\subjclass[msc2000]{Primary 93D15; Secondary  46J20}

\title[Abstract Nyquist criterion and applications]{An abstract Nyquist criterion \\containing old and new results}

 \author{Amol Sasane}
  \address{Mathematics Department,
    London School of Economics,
    Houghton Street, London WC2A 2AE,
    United Kingdom.}
  \email{A.J.Sasane@lse.ac.uk}

\begin{abstract}
  We prove an abstract Nyquist criterion in a general
  set up. As applications, we recover various versions of the Nyquist
  criterion, some of which are new.
\end{abstract}

\maketitle

\section{Introduction}

\noindent Harry Nyquist, in his fundamental paper \cite{Nyq}, gave a
criterion for the stability of a feedback system, which is one of the
basic tools in the frequency domain approach to feedback control.
This test, which is expressed in terms of the winding number around
zero of a certain curve in the complex plane, is well known for finite
dimensional systems; see for example \cite{Vid} or
Theorem~\ref{corollary_nyquist_for_disk_algebra} in this article.
There are several extensions of this test for other classes of systems
as well; see for example \cite{CalDes}, \cite{Dav}, \cite{DecMurSae}.
Thus the problem of obtaining a Nyquist criterion encompassing the
different transfer function classes of systems is a natural one; see
\cite{Log}, \cite[p.65]{Qualn}.

In this article, we will prove an ``abstract Nyquist theorem'', where
we only start with a commutative ring $R$ (thought of as the class of
stable transfer functions of a linear control system) possessing
certain properties, and then give a criterion for the stability of a
closed loop feedback system formed by a plant and a controller (which
have transfer functions that are matrices with entries from the field
of fractions of $R$). We then specialize $R$ to several classes of
stable transfer functions and obtain various versions of the Nyquist
criterion. In the section on applications, we have given references to
the known results; all other results seem to be new.

The article is organized as follows: 
\begin{enumerate}
\item In Section \ref{section_general_assumptions}, we describe the
  basic objects in our abstract set up in which we will prove our
  abstract Nyquist criterion. The starting point will be a commutative
  ring $R$. We will also give a systematic procedure to build the
  other basic objects starting from $R$ in teh case when $R$ is a
  Banach algebra.
\item In section \ref{section_fact_approach}, we will recall the
  standard definitions from the factorization approach to feedback
  control theory.
\item In Section \ref{section_abstract_Nyquist}, we prove our main
  result, the abstract Nyquist criterion, in Theorem~\ref{thm_main}.
\item Finally in various subsections of
  Section~\ref{section_applications}, we recover some old versions of
  the Nyquist criterion as well as obtain new ones, as special instances
  of our abstract Nyquist criterion. 
\end{enumerate}

\section{General setup and assumptions}
\label{section_general_assumptions}

\noindent Our set up is a triple $(R,S, \iota)$, satisfying the
following:
\begin{itemize}
\item[(A1)] $R$ be a unital commutative ring.
\item[(A2)] $S$ is a unital commutative Banach algebra such that
  $R\subset S$.  The invertible elements of $S$ will be denoted by
  $\inv S$.
\item[(A3)] There exists a map $\iota: \inv S \rightarrow G$, where
  $(G,\star)$ is an Abelian group with identity denoted by $\circ$, and
  $\iota$ satisfies
$$
\iota(ab)= \iota (a) \star \iota(b)\quad (a,b \in \inv S).
$$
The function $\iota$ will be called an {\em abstract index}.
\item[(A4)] $x\in R \cap (\inv S)$ is invertible as an element of $R$
  iff $\iota(x)=\circ$.
\end{itemize}

\noindent Typically, one has $R$ available. So the natural question
which arises is: How does one find $S$ and $\iota$ that satisfy
(A1)-(A4)? We outline a systematic procedure for doing this below when
$R$ is a commutative unital complex Banach algebra (or more generally
a full subring of such a Banach algebra; the definition of a full
subring is  recalled below).

\begin{definition}
  Let $R_1, R_2$ be commutative unital rings, and let $R_1$ be a
  subring of $R_2$. Then $R_1$ is said to be a {\em full} subring of $R_2$
  if for every $x\in R_1$ such that $x$ is invertible in $R_2$, there
  holds that $x$ is invertible in $R_1$.
\end{definition}

\subsection{A choice of $\iota$}
\label{subsection_iota}

If $\exp S$ denotes the connected component in $\inv S$ which contains
the identity element of $S$, then we can take $G$ as the (discrete)
group $(\inv S)/(\exp S)$, and $\iota$ can be taken to be the natural
homomorphism $\iota_S$ from $\inv S$ to $(\inv S)/(\exp S)$.  Then
(A3) holds; see \cite[Proposition 2.9]{Dou}.

\subsection{A choice of $S$}
\label{subsection_S}

On the other hand, one possible construction of an $S$ is as follows.
First we recall a definition from \cite{Moh}.

\begin{definition}
  Let $X_R$ denote the maximal ideal space of a unital commutative
  Banach algebra $R$. A closed subset $Y \subset X_R$ is said to
  satisfy the {\em generalized argument principle} for $R$ if whenever
  $a \in R$ and $\log \widehat{a} $ is defined continuously on $Y$,
  then $a$ is invertible in $R$. (Here $\widehat{a}$ denotes the
  Gelfand transform of $a$, $Y$ is equipped with the topology it
  inherits from $X_R$ and $X_R$ has the usual Gelfand topology).
\end{definition}

It was shown in \cite[Theorem~2.2]{Moh} that any $Y$ satisfying the
generalized argument principle is a boundary for $R$ and so it
contains the \v{S}ilov boundary of $R$. Moreover, given any $R$, there
always exists a minimal closed set $Y_R$ of $X_R$ which satisfies the
generalized argument principle for $R$ \cite[Theorem~2.7]{Moh}.

So if we know a set $Y\subset X_R$ that satisfies the generalized
argument principle for $R$, then one can take $S$ to be equal to
$S_Y:=C(Y)$.  The topology on
$C(Y)$ is the one given by the supremum norm.

\begin{lemma}
\label{lemma_gen_1}
Let $R$ be a commutative unital complex Banach algebra, and let
$Y\subset X_R$ satisfy the generalized argument principle for $R$. Let
$S:=S_Y$ and $\iota:=\iota_{S_Y}$ be as described in the previous two
subsections.  Let $f\in \textrm{\em inv }S$. Then $f$ has a continuous
logarithm iff $\iota(f)=\circ$. In particular the triple $(R,S,\iota)$
satisfies {\em (A1)-(A3)} and the `if' part of {\em (A4)}.
\end{lemma}
\begin{proof} Suppose that $f$ has a continuous logarithm. Then
  $f=e^g$ for some $g\in C(S)$. But then by the definition of $\iota$,
  $\iota(f)=\circ$.

  Conversely, suppose that $\iota(f)=\circ$. This means that $f=e^{g}$
  for some $g\in C(S)$. Hence $f$ has a continuous logarithm.

  (A1) is trivial. Given $f\in R$, we see that $\widehat{f}|_Y \in
  C(Y)$. Moreover the map $f\mapsto \widehat{f}|_Y$ is one-to-one
  since $Y$ contains the \v{S}ilov boundary of $R$. Indeed if
  $\widehat{f}|_Y=0$, then we have
$$
\max_{\varphi \in X_R} |\widehat{f}(\varphi)|=\max_{\varphi \in Y}
|\widehat{f}(\varphi)|=0,
$$ 
and so $\widehat{f}\equiv 0$, that is $f=0$. Hence (A2) holds as well.
(A3) follows from the definition of $\iota$. Finally we show (A4)
below.

Suppose that $f\in R \cap \inv S$.  If $\iota(f)=\circ$, then we know
that $f$ has a continuous logarithm on $Y$. But $Y$ satisfies the
generalized argument principle for $R$.  Thus $f$ is invertible as an
element of $R$.
\end{proof}

For the `only if' part, we will need a stronger property on $Y$ than
the generalized argument principle.

\begin{definition}
A closed subset $Y \subset X_R$ is said to satisfy the {\em strong
  generalized argument principle} for $R$ if $a \in R$ is invertible
as an element in $R$ iff $\log \widehat{a} $ is defined continuously
on $Y$.
\end{definition}

\begin{lemma}
Let $R$ be a commutative unital complex Banach algebra, and let
$Y\subset X_R$ satisfy the strong generalized argument principle for
$R$. Let $S:=S_Y$ and $\iota:=\iota_{S_Y}$ be as described in the
previous subsection.   Then  the triple
$(R,S,\iota)$ satisfies {\em (A1)-(A4)}.
\end{lemma}
\begin{proof} (A1)-(A3) and the `if' part of (A4) have been verified
already in Lemma~\ref{lemma_gen_1}. We just verify the `only if'
part of (A4). So suppose that $f\in R \cap \inv C(Y)$ and that $f$ is
invertible as an element of $R$. Then $f$ has a continuous logarithm
on $Y$, and so $\iota(f)=\circ$, again by Lemma~\ref{lemma_gen_1}.
\end{proof}

In Subsection \ref{subsection_disk_algebra} and \ref{subsection_AP},
in the case of the disk algebra $A(\mD)$ and the analytic almost
periodic algebra $AP^+$, we will see that our choices of $S$ and
$\iota$ are precisely of the type described above.

\section{Feedback stabilization}
\label{section_fact_approach}

\noindent We recall the following definitions from the factorization
approach to control theory.

\begin{definition}
  The field of fractions of $R$ will be denoted by $\mF(R)$. Let $P
  \in (\mF(R))^{p\times m}$ and let $P=ND^{-1}$, where $N,D$ are
  matrices with entries from $R$. Here $D^{-1}$ denotes a matrix with
  entries from $\mF(R)$ such that $DD^{-1}=D^{-1} D=I$. The
  factorization $P=ND^{-1}$ is called a {\em right coprime
    factorization of} $P$ if there exist matrices $X, Y$ with entries
  from $R$ such that $X N + Y D=I_m$.  Similarly, a factorization
  $P=\widetilde{D}^{-1}\widetilde{N}$, where
  $\widetilde{N},\widetilde{D}$ are matrices with entries from $R$, is
  called a {\em left coprime factorization of} $P$ if there exist
  matrices $\widetilde{X}, \widetilde{Y}$ with entries from $R$ such
  that $\widetilde{N} \widetilde{X}+\widetilde{D} \widetilde{Y}=I_p$.
  Given $P \in (\mF(R))^{p\times m}$ with right and left
  factorizations
$$
P=N D^{-1}
\quad \textrm{and} \quad
P= \widetilde{D}^{-1} \widetilde{N},
$$
respectively, we introduce the following matrices with entries from
$R$:
$$
G_P=\left[ \begin{array}{cc} N \\ D \end{array} \right]
\quad \textrm{and} \quad
\widetilde{G}_P=\left[ \begin{array}{cc}
-\widetilde{N}& \widetilde{D}  \end{array} \right] .
$$
We denote by $\mS(R,p, m)$ the set of all $P\in (\mF(R))^{p\times m}$
that possess a right coprime factorization and a left coprime
factorization.
\end{definition}

Given $P \in (\mF(R))^{p\times m}$ and $C\in (\mF(R))^{m\times p}$,
define the {\em closed loop transfer function}
$$
H(P,C):= \left[\begin{array}{cc} P \\ I \end{array} \right]
(I-CP)^{-1} \left[\begin{array}{cc} -C & I \end{array} \right] \in
(\mF(R))^{(p+m)\times (p+m)}.
$$
$C$ is said to {\em stabilize} $P$ if $H(P,C) \in R^{(p+m)\times
  (p+m)}$, and $P$ is called {\em stabilizable} if $ \{ C \in
(\mF(R))^{m\times p} : H(P,C) \in R^{(p+m)\times (p+m)}\} \neq
\emptyset$.  If $ P\in \mS(R,p,m)$, then $P$ is a {\em stabilizable};
see for example \cite[Chapter 8]{Vid}. Thus
$$
\mS(R,p,m)
= \left\{ P\in (\mF(R))^{p\times m} \bigg|
\begin{array}{ll}
   \exists C\in (\mF(R))^{m\times p} \textrm{ such that} \\
   H(P,C)\in R^{(p+m)\times (p+m)}
\end{array}
\right\}.
$$
It was shown in \cite[Theorem 6.3]{Qua} that if the ring $R$ is
projective free, then every stabilizable $P$ admits a right coprime
factorization and a left coprime factorization.

We will use the following in order to prove our main result in the
next section.

\begin{lemma}
\label{lemma_fth}
Suppose that $F\in R^{m\times m}$. Then $F$ is invertible as an
element of $ R^{m\times m}$ iff $\det F \in \textrm{\em inv } S$ and
$\iota( \det F)=\circ$.
\end{lemma}
\begin{proof} Using Cramer's rule, we see that $F$ is invertible as an element of
$R^{m\times m}$ iff $\det F$ is invertible as an element of $R$. The
result now follows from (A4).
\end{proof}

\section{Abstract Nyquist criterion}
\label{section_abstract_Nyquist}

\begin{theorem}
\label{thm_main} 
Let $(${\em A}$1)$-$(${\em A}$4)$ hold. Suppose that $P\in \mS(R, p,
m)$ and that $C \in \mS(R, m, p)$. Moreover, let $P=N_P D_P^{-1}$ be a
right coprime factorization of $P$, and let $C=\widetilde{D}_C^{-1}
\widetilde{N}_C $ be a left coprime factorization of $C$. Then the
following are equivalent:
\begin{enumerate}
\item $C$ stabilizes $P$.
\item
\begin{enumerate}
\item $\det (I-CP), \det D_P, \det \widetilde{D}_C \in \textrm{\em inv
  } S$ and
\item $\iota(\det (I-CP))\star \iota(\det D_P) \star \iota(\det
  \widetilde{D}_C)=\circ$.
\end{enumerate}
\end{enumerate}
\end{theorem}
\begin{proof} We note that
\begin{eqnarray*}
H(P,C)&=& \left[\begin{array}{cc} P \\ I \end{array} \right]
(I-CP)^{-1} \left[\begin{array}{cc} -C & I \end{array} \right]
\\
&=&
\left[\begin{array}{cc}N_P D_P^{-1} \\ I \end{array} \right]
(I-\widetilde{D}_C^{-1} \widetilde{N}_C N_P D_P^{-1})^{-1}
\left[\begin{array}{cc} -\widetilde{D}_C^{-1} \widetilde{N}_C& I \end{array} \right]
\\
&=&
\left[\begin{array}{cc}N_P  \\ D_P \end{array} \right]
(\widetilde{D}_C D_P -\widetilde{N}_C N_P )^{-1}
\left[\begin{array}{cc} - \widetilde{N}_C&  \widetilde{D}_C\end{array} \right]
\\
&=& G_P (\widetilde{G}_C G_P)^{-1} \widetilde{G}_C.
\end{eqnarray*}
So if $(\widetilde{G}_C G_P)^{-1} \in R^{p \times p}$, then $H(P,C)\in
R^{(p+m)\times (p+m)}$ . Conversely, using the fact that there exist
matrices $\Theta$ and $\widetilde{\Theta}$ with $R$ entries such that
$\Theta G_P=I$ and $\widetilde{G}_C\widetilde{\Theta}=I$, it follows
from the above that if $H(P,C)\in R^{(p+m)\times (p+m)}$, then
$(\widetilde{G}_C G_P)^{-1} \in R^{p \times p}$. So $C$ stabilizes $P$
iff $(\widetilde{G}_C G_P)^{-1} \in R^{p \times p}$. We will use this
fact below.

\medskip

\noindent (1)$\Rightarrow$(2): Suppose that $C$ stabilizes $P$. Then
$(\widetilde{G}_C G_P)^{-1} \in R^{p \times p}$. So $\det
(\widetilde{G}_C G_P)$ is invertible as an element of $R$. By (A4), it
follows that $\det (\widetilde{G}_C G_P)$ is invertible as an element
of $S$ and $\iota(\det (\widetilde{G}_C G_P))=\circ$. But
$$
 \widetilde{G}_C G_P =  \widetilde{D}_C D_P -\widetilde{N}_C N_P  =
\widetilde{D}_C(I-CP) D_P.
$$
Thus $\det (\widetilde{G}_C G_P) = (\det \widetilde{D}_C) \cdot (\det
(I-CP))\cdot (\det D_P)$ and so $(\det \widetilde{D}_C) \cdot (\det
(I-CP))\cdot (\det D_P)\in \inv S$. Hence  $\det
\widetilde{D}_C$, $\det (I-CP)$, $\det D_P$ are each invertible
elements of $S$. From (A3) we obtain 
$$
\circ=\iota( \det
(\widetilde{G}_C G_P) ) = \iota(\det \widetilde{D}_C)\star \iota
(\det(I-CP)) \star \iota(\det D_P).
$$

\medskip

\noindent (2)$\Rightarrow$(1): Suppose that $\det (I-CP), \det D_P,
\det \widetilde{D}_C \in \inv  S$ and that 
$$
\iota(\det (I-CP))\star \iota(\det D_P)\star  \iota(\det \widetilde{D}_C)=\circ.
$$
Then retracing the above steps in the reverse order, we see that $\det
(\widetilde{G}_C G_P)$ is invertible in $S$, and moreover,
$$
\iota( \det (\widetilde{G}_C G_P) ) =
\iota(\det \widetilde{D}_C)\star  \iota (\det(I-CP)) \star\iota(\det D_P)=\circ.
$$
From (A4) it follows that $\det (\widetilde{G}_C G_P)$ is invertible
as an element of $R$. Thus $\widetilde{G}_C G_P$ is invertible as an
element of $R^{p\times p}$. Consequently $C$ stabilizes $P$.
\end{proof}

\section{Applications}
\label{section_applications}

Now we specialize $R$ to several classes of stable transfer functions
and obtain various versions of the Nyquist criterion. In particular,
we begin with Subsection~\ref{subsection_disk_algebra}, where we
recover the classical Nyquist criterion.

\subsection{The disk algebra}
\label{subsection_disk_algebra}
Let
$$
\mD:= \{ z\in \mC: |z| <1\},\quad 
\overline{\mD}:= \{ z\in \mC: |z| \leq 1\}, \quad 
\mT:= \{ z\in \mC: |z|=1\}.
$$
The {\em disk algebra} $A(\mD)$ is the set of all functions $f:
\overline{\mD} \rightarrow \mC$ such that $f$ is holomorphic in $\mD$
and continuous on $\overline{\mD}$.  Let $C(\mT)$ denote the set of
complex-valued continuous functions on the unit circle $\mT$. For each
$f\in \inv C(\mT)$, we can define the {\em winding number} ${\tt
  w}(f)\in \mZ$ of $f$ as follows:
$$
{\tt w}(f)= \frac{1}{2\pi}(\Theta(2\pi)- \Theta(0)),
$$
where $\Theta:[0,2\pi] \rightarrow \mR$ is a continuous function such
that
$$
f(e^{it}) =|f(e^{it})| e^{i \Theta(t)}, \quad t \in [0,2\pi].
$$
The existence of such a $\Theta$ can be proved; see \cite[Lemma
4.6]{Ull}. Also, it can be checked that ${\tt w}$ is well-defined and
integer-valued. Geometrically, ${\tt w}(f)$ is the number of times the
curve $t \mapsto f(e^{it}): [0,2\pi] \rightarrow \mC$ winds around the
origin in a counterclockwise direction.  Also, \cite[Lemma
4.6.(ii)]{Ull} shows that the map ${\tt w}: \inv C(\mT) \rightarrow
\mR$ is locally constant. Here the local constancy of ${\tt w}$ means
continuity relative to the discrete topology on $\mR$, while $C(\mT)$
is equipped with the usual $\sup$-norm. 

\begin{lemma}
  Let
\begin{eqnarray*}
  R&=& \textrm{\em a unital full subring of }A(\mD),\\
  S&:=& C(\mT), \\
  G&:=& \mZ, \\
  \iota&:=&{\tt w}.
\end{eqnarray*}
Then {\em (A1)-(A4)} are satisfied.
\end{lemma}
\begin{proof} (A1) and (A2) are clear. (A3) is evident from the
definition of {\tt w}. Finally, we will show below that (A4) holds.

Suppose that $f \in R \cap (\inv C(\mT))$ is invertible as an element
of $R$. Then obviously $f$ is also invertible as an element of
$A(\mD)$.  Hence it has no zeros or poles in $\overline{\mD}$.  For
$r\in (0,1)$, define $f_r \in A(\mD)$ by $f_r(z)=f(rz)$ ($z\in
\overline{\mD}$).  Then $f_r$ also has no zeros or poles in
$\overline{\mD}$, and has a holomorphic extension across $\mT$. From
the Argument Principle (applied to $f_r$), it follows that ${\tt
  w}(f_r)=0$. But $\|f_r-f\|_\infty \rightarrow 0$ as $r\nearrow 1$.
Hence ${\tt w}(f)= \displaystyle \lim_{r\rightarrow 1} {\tt w}(f_r)=
\displaystyle \lim_{r\rightarrow 1} 0=0$.

Suppose, conversely, that $f\in R \cap (\inv C(\mT))$ is such that
${\tt w}(f)=0$. For all $r\in (0,1)$ sufficiently close to $1$, we
have that $f_r\in \inv C(\mT)$. Also, by the local constancy of ${\tt
  w}$, for $r$ sufficiently close to $1$, ${\tt w}(f_r)= {\tt
  w}(f)=0$. By the Argument principle, it then follows that $f_r$ has
no zeros in $\overline{\mD}$. Equivalently, $f$ has no zeros in $r
\overline{\mD}$. But letting $r\nearrow 1$, we see that $f$ has no
zeros in $\mD$. Moreover, $f$ has no zeros on $\mT$ either, since
$f\in \inv C(\mT)$. Thus $f$ has no zeros in $\overline{\mD}$.
Consequently, we conclude that $f$ is invertible as an element of
$A(\mD)$. (Indeed, $f$ is invertible as an element of
$C(\overline{\mD}$, and it is also then clear that this inverse is
holomorphic in $\mD$.) Finally, since $R$ is a full subring of
$A(\mD)$, we can conclude that $f$ is invertible also as an element of
$R$.
\end{proof}

\noindent Besides $A(\mD)$ itself, some other examples of such $R$ are:
\begin{enumerate}
\item $\calP$, the set of all polynomial functions in the variable $z
  \in \mC$.
\item $RH^\infty(\mD)$, the set of all rational functions without
  poles in $\overline{\mD}$.
\item The Wiener algebra $W^{+}(\mD)$ of all functions $f\in A(\mD)$
  that have an absolutely convergent Taylor series about the origin:

$\displaystyle \sum_{n=0}^\infty |f_n|<+\infty$, where $f(z)=
\displaystyle \sum_{n=0}^\infty f_n z^n$ ($z\in \mD$).
\item $\partial^{-n}H^\infty(\mD)$, the set of $f:\mD \rightarrow \mC$ such
  that $f,f^{(1)},f^{(2)},\dots,f^{(n)}$ belong to $H^\infty(\mD)$. Here
  $H^\infty(\mD)$ denotes the Hardy algebra of all bounded and holomorphic
  functions on $\mD$.
\end{enumerate}

\noindent An application of our main result (Theorem~\ref{thm_main})
yields the following Nyquist criterion. We note that invertibility of
$f$ in $C(\mT)$ just means that $f$ belongs to $C(\mT)$ and
it has no zeros on $\mT$.

\begin{corollary}
\label{corollary_nyquist_for_disk_algebra}
  Let $R$ be a unital full subring of $A(\mD)$. Let $P\in \mS(R,
  p, m)$ and $C \in \mS(R, m, p)$. Moreover, let $P=N_P D_P^{-1}$ be
  a right coprime factorization of $P$, and $C=\widetilde{D}_C^{-1} \widetilde{N}_C $ be a
  left coprime factorization of $C$. Then the following are
  equivalent:
\begin{enumerate}
\item $C$ stabilizes $P$.
\item 
\begin{enumerate}
\item
$\det (I-CP)$ belongs to $C(\mT)$,  
\item $\det (I-CP)$, $\det D_P$, $\det \widetilde{D}_C$ have no zeros
  on $\mT$, and 
\item ${\tt w}(\det (I-CP))+{\tt w}(\det D_P)+ {\tt w}(\det
  \widetilde{D}_C)=0$.
\end{enumerate}
\end{enumerate}
\end{corollary}

\noindent It can be shown that $Y=\mT$ satisfies the generalized
argument principle for $A(\mD)$; see \cite[Corollary~1.25]{Moh}.
Moreover, we know that if a function in $A(\mD)$ is invertible, then
by considering the map $r \mapsto f_r|_{\mT}:[0,1] \rightarrow \inv
C(\mT)$, we see that $f$ belongs to the connected component of $\inv
C(\mT)$ that contains $1$. So it is of the form $f|_{\mT}=e^{g}$ for
some $g\in C(\mT)$. Hence $f|_{\mT}$ has a continuous logarithm on
$\mT$. So we can take $S=C(\mT)$.  Moreover, if $\exp C(\mT)$ denotes
the connected component in $\inv C(\mT)$ which contains the constant
function $1$ on $\mT$, then $G= (\inv C(\mT)/(\exp C(\mT))$ is
isomorphic to $\mZ$ (see for example \cite[Corollary~2.20]{Dou}), and
$\iota$ can be taken as the the natural homomorphism from $\inv
C(\mT)$ to $\mZ$ given by the winding number.

\begin{remark}
  $\calP$, $RH^\infty(\mD)$ are projective free rings since they are
  both Bezout domains. Also $A(\mD)$, $W^+(\mD)$, or
  $\partial^{-n}H^\infty(\mD)$ are projective free rings, since their
  maximal ideal space is $\overline{\mD}$, which is contractible; see
  \cite{BruSas}. Thus if $R$ is one of $\calP$, $RH^\infty(\mD)$,
  $A(\mD)$, $W^{+}(\mD)$ or $\partial^{-n}H^\infty(\mD)$, then the set
  $\mS(R, p, m)$ of plants possessing a left and a right coprime
  factorization coincides with the class of plants that are
  stabilizable by \cite[Theorem~6.3]{Qua}.

  The result in Corollary \ref{corollary_nyquist_for_disk_algebra} was
  known in the special cases when $R$ is $\calP$, $RH^\infty(\mD)$ or
  $A(\mD)$; see \cite{Vid}.
\end{remark}

\subsection{Almost periodic functions}
\label{subsection_AP}

The algebra $AP$ of complex valued (uniformly) {\em almost periodic
  functions} is the smallest closed subalgebra of $L^\infty(\mR)$ that
contains all the functions $e_\lambda := e^{i \lambda y}$. Here the
parameter $\lambda$ belongs to $\mR$.  For any $f\in AP$, its {\em
  Bohr-Fourier series} is defined by the formal sum
\begin{equation}
\label{eq_BFs}
\sum_{\lambda} f_\lambda e^{i  \lambda y} , \quad y\in \mR,
\end{equation}
where
$$
f_\lambda:= \lim_{N\rightarrow \infty} \frac{1}{2N}
\int_{[-N,N]}   e^{-i \lambda y} f(y)dy, \quad
\lambda \in \mR,
$$
and the sum in \eqref{eq_BFs} is taken over the set $
\sigma(f):=\{\lambda \in \mR\;|\; f_\lambda \neq 0\}$, called the {\em
  Bohr-Fourier spectrum} of $f$. The Bohr-Fourier spectrum of every
$f\in AP$ is at most a countable set.

The {\em almost periodic Wiener algebra} $APW$ is defined as the set
of all $AP$ such that the Bohr-Fourier series \eqref{eq_BFs} of $f$
converges absolutely. The almost periodic Wiener algebra is a Banach
algebra with pointwise operations and the norm $\|f\|:=
\displaystyle\sum_{\lambda \in \mR} |f_\lambda|$.  Set
\begin{eqnarray*}
  AP^+&=& \{ f \in AP \;|\; \sigma(f) \subset [0,\infty)\}\\
  APW^+&=& \{ f \in APW \;|\; \sigma(f) \subset [0,\infty)\}.
\end{eqnarray*}
Then $AP^+$ (respectively $APW^+$) is a Banach subalgebra of $AP$
(respectively $APW$). For each $f\in \inv AP$, we can define the {\em
  average winding number} $w(f)\in \mR$ of $f$ as follows:
$$
w(f)= \lim_{T \rightarrow \infty} \frac{1}{2T}
\bigg( \arg (f(T))-\arg(f(-T))\bigg).
$$
See \cite[Theorem 1, p. 167]{JesTor}.

\begin{lemma}
\label{lemma_AP_111}
Let
\begin{eqnarray*}
  R&:=& \textrm{\em a unital full subring of }AP^+\\
  S&:=& AP, \\
  G&:=& \mR,\\
  \iota&:=& w.
\end{eqnarray*}
Then {\em (A1)-(A4)} are satisfied.
\end{lemma}
\begin{proof} (A1) and (A2) are clear. (A3) follows from the
  definition of $w$. Finally, (A4) follows from
  \cite[Theorem~1,~p.776]{CalDes} which says that $f\in AP^+$
  satisfies
\begin{equation}
\label{cc_AP+}
\inf_{\textrm{Im}(s)\geq 0} |f(s)| >0
\end{equation}
iff $\displaystyle \inf_{y\in \mR}|f(y)| >0$ and $w(f)=0$.  But
$$
\displaystyle \inf_{y\in \mR}|f(y)| >0
$$
is equivalent to $f$ being an invertible element of $AP$ by the corona
theorem for $AP$ (see for example \cite[Exercise~18,~p.24]{Gam}). Also
the equivalence of \eqref{cc_AP+} with that of the invertibility of
$f$ as an element of $AP^+$ follows from the Arens-Singer corona
theorem for $AP^+$ (see for example \cite[Theorems~3.1,~4.3]{Bot}).
Finally, the invertibility of $f \in R$ in $R$ is equivalent to the
invertibility of $f$ as an element of $AP^+$ since $R$ is a full
subring of $AP^+$.
\end{proof}

\begin{remark}
\label{examples_of_R_in_AP+}
Specific examples of such $R$ are $AP^+$ and $APW^+$. More generally,
let $\Sigma \subset [0,+\infty)$ be an {\em additive semigroup} (if
$\lambda,\mu \in \Sigma $, then $\lambda+\mu \in \Sigma$) and suppose
$0 \in \Sigma$.  Denote
\begin{eqnarray*}
  AP_\Sigma&=& \{ f \in AP \;|\; \sigma(f) \subset \Sigma\}\\
  APW_\Sigma&=& \{ f \in APW \;|\; \sigma(f) \subset \Sigma\}.
\end{eqnarray*}
Then $AP_\Sigma$ (respectively $APW_\Sigma$) is a unital Banach
subalgebra of $AP^+$ (respectively $APW^+$). Let
$\overline{Y_{\Sigma}}$ denote the set of all maps $\theta:\Sigma
\rightarrow [0,+\infty]$ such that $ \theta(0)=0$ and $\theta
(\lambda+\mu)= \theta (\lambda)+ \theta(\mu)$ for all $\lambda, \mu
\in \Sigma$.  Examples of such maps $\theta$ are the following. If
$y\in [0,+\infty)$, then $\theta_y$, defined by $\theta_y(\lambda)=
\lambda y$, $\lambda \in \Sigma$, belongs to $\overline{Y_{\Sigma}}$.
Another example is $\theta_{\scriptscriptstyle \infty}$, defined as follows:
$$
\theta_{\scriptscriptstyle \infty}(\lambda)= \left\{ \begin{array}{ll}
0 & \textrm{if } \lambda =0,\\
+\infty & \textrm{if } \lambda \neq 0.
\end{array}\right.
$$
So in this way we can consider $[0,+\infty]$ as a subset of
$\overline{Y_{\Sigma}}$.

The results \cite[Proposition~4.2, Theorem 4.3]{Bot} say that if
$\overline{Y_{\Sigma}} \subset [0,+\infty]$, and $f\in AP_\Sigma$
(respectively $APW_\Sigma$), then $f \in \inv AP_\Sigma$ (respectively
$\in \inv APW_\Sigma$) iff \eqref{cc_AP+} holds. So in this case
$AP_\Sigma$ and $APW_\Sigma$ are unital full subalgebras of $AP^+$.
\end{remark}

An application of our main result (Theorem \ref{thm_main}) yields the
following Nyquist criterion. We note that invertibility of $f$ in $AP$
just means that $f$ belongs to $AP$ and is bounded away from zero on
$\mR$ again by the corona theorem for $AP$.

\begin{corollary}
\label{corollary_nyquist_for_AP}
  Let $R$ be a unital full subring of $AP^+$. Let $P\in \mS(R, p, m)$
  and $C \in \mS(R, m, p)$. Moreover, let $P=N_P D_P^{-1}$ be a right
  coprime factorization of $P$, and $C=\widetilde{D}_C^{-1}
  \widetilde{N}_C $ be a left coprime factorization of $C$. Then the
  following are equivalent:
\begin{enumerate}
\item $C$ stabilizes $P$.
\item 
\begin{enumerate}
\item $\det (I-CP)$ belongs to $AP$, 
\item $\det (I-CP)$, $\det D_P$, $\det
  \widetilde{D}_C$ are bounded away from $0$ on $\mR$, 
\item $ w(\det
  (I-CP))+ w(\det D_P)+ w(\det \widetilde{D}_C)=0$.
\end{enumerate}
\end{enumerate}
\end{corollary}

\noindent Finally, in the case of the analytic almost periodic algebra
$AP^+$, we show below that the choices of $S$ and $\iota$ are
precisely of the type described in Subsections~\ref{subsection_iota}
and \ref{subsection_S}.  Let $\mR_B$ denote the Bohr compactification
of $\mR$. Then $X_{AP^+}$ contains a copy of $\mR_B$ (since
$X_{AP}=\mR_B$, and $AP^+\subset AP$), and we show below that
$Y:=\mR_B$ satisfies the strong generalized argument principle for
$AP^+$. Thus we can take $S=C(\mR_B)=AP$, and we will also show that
the $\iota_{AP}$ coincides with the average winding number defined
above.

\begin{lemma}
$\mR_B$ satisfies the strong generalized argument principle for $AP^+$.
\end{lemma}
\begin{proof} First of all, suppose that $f\in AP^+$ has a continuous
  logarithm on $\mR_B$. Then $f=e^g $ for some $g\in C(\mR_B)=AP$. But
  then since $g\in AP$, we have that $\textrm{Im}(g)$ is bounded on
  $\mR$.
\begin{eqnarray*}
w(f)&=& \lim_{T\rightarrow \infty} \frac{1}{2R}
\bigg( \arg (f(T))-\arg(f(-T))\bigg)
\\
&=&
\lim_{T \rightarrow \infty} \frac{1}{2T}
\bigg( \textrm{Im} (g(T))-\textrm{Im}(g(-T))\bigg)
=0
.
\end{eqnarray*}
But by (A4) (shown in Lemma~\ref{lemma_AP_111}), it follows that $f$
is invertible as an element of $AP^+$.

Conversely, suppose that 
$$
f=\displaystyle \sum_{n=1}^\infty f_n
e^{i\lambda_n \cdot}
$$ 
is invertible as an element of $AP^+$.  Consider the map $\Phi: [0,1] \rightarrow \inv AP$ given by
$\Phi(t)=f(\cdot-i\log(1-t))$ if $t\in [0,1)$ and $\Phi(1)=f_0$. Thus
$\widehat{f}|_{\mR_B}$ belongs to the connected component of $\inv AP$
that contains the constant function $1$. Hence
$\widehat{f}|_{\mR_B}=e^{g}$ for some $g\in C(\mR_B)$. This shows that
$\widehat{f}$ has a continuous logarithm on $\mR_B$.
\end{proof}

Moreover, $\iota_{AP}$ coincides with the average winding number.
Indeed, the result \cite[Theorem 1, p. 167]{JesTor} says that if $f\in
\inv AP$, then there exists a $g\in AP$ such that
$\arg f(t)=w(f) t +g(t)$ ($t\in \mR$). Hence
$$
f=|f| e^{i(w(f)t+g)}=e^{\log |f| +i(w(f)t+g)}=e^{\log |f|+ig}
e^{iw(f)t}.
$$
Since $\log |f|+ig\in AP$, it follows that
$\iota_{AP}(f)=\iota_{AP}(e^{iw(f)t})$. But now with the association
$\iota_{AP}(e^{iw(f)t})\leftrightarrow w(f)$, we see that the maps
$\iota_{AP}$ and $w$ are the same.

So $AP$ and $w$ are precisely $S_Y$ and $\iota_{C(Y)}$, respectively,
described in Subsections~\ref{subsection_iota} and \ref{subsection_S}
when $Y=\mR_B$.

\begin{remark}
  It was shown in \cite{BruSas} that $AP^+$ and $APW^+$ are projective
  free rings.  Thus if $R=AP^+$ or $APW^+$, then the set $\mS(R, p,
  m)$ of plants possessing a left and a right coprime factorization
  coincides with the class of plants that are stabilizable by
  \cite[Theorem~6.3]{Qua}.

  Corollary \ref{corollary_nyquist_for_AP} was known in the special
  case when $R=APW^+$; see \cite{CalDes}.
\end{remark}

\subsection{Algebras of Laplace transforms of measures without a
  singular nonatomic part}

Let $\mC_{+}:=\{s\in \mC\;|\; \textrm{Re}(s)\geq 0\}$ and let
$\calA^+$ denote the Banach algebra

$
\calA^+=\left\{ s (\in \mC_{+}) \mapsto \widehat{f_a}(s)
  +\displaystyle \sum_{k=0}^\infty f_k e^{- s t_k} \; \bigg| \;
\begin{array}{ll}
f_a \in L^{1}(0,\infty), \;(f_k)_{k\geq 0} \in \ell^{1},\\
0=t_0 <t_1 ,t_2 , t_3, \dots
\end{array} \right\}
$

\noindent equipped with pointwise operations and the norm:

$
\|F\|=\|f_a\|_{\scriptscriptstyle L^{1}} +
\|(f_k)_{k\geq 0}\|_{\scriptscriptstyle \ell^1},
\;\; F(s)=\widehat{f_a}(s) +\displaystyle\sum_{k=0}^\infty f_k e^{-st_k}\;\;(s\in \mC_+).
$

\noindent Here $\widehat{f_a}$ denotes the {\em Laplace transform of}
$f_a$, given by
$$
\widehat{f_a}(s)=\displaystyle \int_0^\infty e^{-st} f_a(t)
dt, \quad s \in \mC_+.
$$
Similarly, define the Banach algebra $\calA$ as follows (\cite{GohFel}):

$
\!\!\!\!\!\!\calA\!=\!\left\{ iy (\in i\mR) \mapsto \widehat{f_a}(iy)
  +\!\!\!\displaystyle\sum_{k=-\infty}^\infty f_k e^{- iy t_k} \bigg|
\begin{array}{ll}
f_a \in L^{1}(\mR), \;(f_k)_{k\in \mZ } \in \ell^{1},\\
\dots, t_{-2}, t_{-1}<\!0\!=\!t_0\! <t_1 ,t_2 ,  \dots
\end{array} \!\!\!\right\}
$

\noindent equipped with pointwise operations and the norm:

$
\|F\|=\|f_a\|_{\scriptscriptstyle L^{1}} + \|(f_k)_{k\in
  \mZ}\|_{\scriptscriptstyle \ell^1}, \;\; F(iy):=\widehat{f_a}(iy)
+\displaystyle\sum_{k=-\infty}^\infty f_k e^{-iy t_k}\;\;(y\in \mR).
$

\noindent Here $\widehat{f_a}$ is the {\em Fourier transform of}
$f_a$, $
\widehat{f_a}(iy)= \displaystyle \int_{-\infty}^\infty e^{-iyt} f_a(t)
dt$, ($y \in \mR$). 

\noindent It can be shown that $\widehat{L^{1}(\mR)}$ is an ideal of $\calA$.

\noindent For $ F=\widehat{f_a} +\displaystyle\sum_{k=-\infty}^\infty
f_k e^{-i\cdot t_k}\in \calA$, we set
$F_{AP}(iy)=\displaystyle\sum_{k=-\infty}^\infty f_k e^{-iy t_k}$
($y\in \mR$).

\noindent If $F =\widehat{f_a}+F_{AP} \in \inv \calA$, then it can be shown that
$F_{AP}(i \cdot) \in \inv AP$ as follows. First of all, the maximal
ideal space of $\calA$ contains a copy of the maximal ideal space of
$APW$ in the following manner: if $\varphi \in M(APW)$, then the map
$\Phi: \calA \rightarrow \mC$ defined by
$\Phi(F)=\Phi(\widehat{f_a}+F_{AP})=\varphi(F_{AP}(i\cdot))$, ($F \in
\calA$), belongs to $M(\calA)$. So if $F$ is invertible in $\calA$, in
particular for every $\Phi$ of the type describe above, $0 \neq
\Phi(F)=\varphi(F_{AP}(i\cdot))$. Thus by the elementary theory of
Banach algebras, $F_{AP}(i\cdot)$ is an invertible element of $AP$.

Moreover, since $\widehat{L^1(\mR)}$ is an ideal in $\calA$,
$F_{AP}^{-1}\widehat{f_a}$ is the Fourier transform of a function in
$L^{1}(\mR)$, and so the map $y \mapsto 
1+(F_{AP}(iy))^{-1}\widehat{f_a}(iy)=\frac{F(iy)}{F_{AP}(iy)}$ has a
well-defined winding number ${\tt w}$ around $0$. Define $W: \inv
\calA \rightarrow \mR\times \mZ$ by $ W (F)= (w(F_{AP}), {\tt
  w}(1+F_{AP}^{-1} \widehat{f_a}) )$, where
$F=\widehat{f_a}+F_{AP} \in \inv \calA$, and
$$
\begin{array}{ll}
w(F_{AP})
:=
\displaystyle \lim_{R \rightarrow \infty} \frac{1}{2R}
\bigg( \arg \big(F_{AP}(iR)\big)-\arg\big(F_{AP}(-iR)\big)\bigg),
\\
{\tt w}(1+F_{AP}^{-1} \widehat{f_a})
:=
\displaystyle \frac{1}{2\pi}\bigg( \arg \big(1+(F_{AP}(iy)\big)^{-1} \widehat{f_a}(iy)  )
\bigg|_{y=-\infty}^{y=+\infty}\bigg).
\end{array}
$$

\begin{lemma}
\label{lemma_inv_A}
$F=\widehat{f_a}+F_{AP} \in \calA$ is invertible iff for all $y\in
\mR$, $F(iy) \neq 0$ and $\displaystyle \inf_{y\in \mR} |F_{AP}(iy)| >0$ .
\end{lemma}
\begin{proof} The `only if' part is clear. We simply show the `if'
  part below.

  Let $F=\widehat{f_a}+F_{AP} \in \calA$ be such that
$$
\inf_{y\in \mR} |F_{AP}(iy)| >0.
$$
Thus $F(i\cdot)$ is invertible as an element of $AP$. Hence
$F=F_{AP}(1+\widehat{f_{a}} F_{AP}^{-1})$ and so it follows that $ (
1+\widehat{f_{a}} F_{AP}^{-1})(iy) \neq 0$ for all $y \in \mR$.  But
by the corona theorem for
$$
\calW:=\widehat{L^1(\mR)}+\mC
$$
(see \cite[Corollary~1,~p.109]{GelRaiShi}), it follows that
$1+\widehat{f_{a}} F_{AP}^{-1}$ is invertible as an element of $\calW$
an in particular, also as an element of $\calA$. This completes the
proof.
\end{proof}

\begin{lemma}
Let
\begin{eqnarray*}
  R&:=&\textrm{\em a unital full subring of } \calA^+, \\
  S&:=& \calA, \\
  G&:=& \mR\times \mZ,\\
  \iota&:=& W.
\end{eqnarray*}
Then {\em (A1)-(A4)} are satisfied.
\end{lemma}
\begin{proof} (A1) and (A2) are clear. (A3) follows from the
  definition of $i$ as follows. Let $F=\widehat{f_a}+F_{AP}$ and
  $G=\widehat{g_a}+G_{AP}$. Then we have
$$
w(F_{AP}G_{AP})=w(F_{AP})+w(G_{AP})
$$
from the definition of $w$. Thus
\begin{eqnarray*}
W(FG)&=&W(( \widehat{f_a}+F_{AP})(\widehat{g_a}+G_{AP})\\
&=& W(\widehat{f_a}\widehat{g_a}+\widehat{f_a}G_{AP}+ \widehat{g_a}F_{AP}+ F_{AP}G_{AP})\\
&=& ({\tt w}(1+(F_{AP}G_{AP})^{-1}(\widehat{f_a}\widehat{g_a}+\widehat{f_a}G_{AP}+ \widehat{g_a}F_{AP}), w(F_{AP}G_{AP}))\\
&=& ({\tt w}((1+F_{AP}^{-1} \widehat{f_a})(1+G_{AP}^{-1} \widehat{g_a})) , w(F_{AP})+w(G_{AP}))\\
&=& ({\tt w}(1+F_{AP}^{-1} \widehat{f_a})+{\tt w}(1+G_{AP}^{-1} \widehat{g_a}),w(F_{AP})+w(G_{AP}))\\
&=& W(\widehat{f_a}+F_{AP})+W(\widehat{g_a}+G_{AP}).
\end{eqnarray*}
So (A3) holds.

Finally we check that (A4) holds. Suppose that $
F=\widehat{f_{a}}+F_{AP}$ belonging to $ (\calA^{+})\cap (\inv
\calA)$, is such that $W(F)=0$. Since $F$ is invertible in $\calA$, it
follows that $F_{AP}(i\cdot)$ is invertible as an element of $AP$. But
$w(F_{AP})=0$, and so $F_{AP}(i\cdot)\in AP^+$ is invertible as an
element of $AP^+$. But this implies that $1+F_{AP}^{-1}
\widehat{f_{a}}$ belongs to the Banach algebra
$$
\calW^{+}:=\widehat{L^{1}(0,\infty)}+\mC.
$$
Moreover, it is bounded away from $0$ on $i\mR$ since
$$
1+F_{AP}^{-1} \widehat{f_{a}} =\frac{F}{F_{AP}},
$$
and $F$ is bounded away from zero on $i\mR$. Moreover ${\tt
  w}(1+F_{AP}^{-1} \widehat{f_{a}})=0$, and so it follows that
$1+F_{AP}^{-1} \widehat{f_{a}}$ is invertible as an element of
$\calW^{+}$, and in particular in $\calA^+$. Since
$F=(1+F_{AP}^{-1} \widehat{f_a})F_{AP}$ and we have shown that both
$(1+F_{AP}^{-1} \widehat{f_a})$ as well as $F_{AP}$ are invertible as
elements of $\calA^+$, it follows that $F$ is invertible in $\calA^+$.
\end{proof}

An example of such a $R$ (besides $\calA^+$) is the algebra
$$
\widehat{L^{1}(0,+\infty)}+APW_\Sigma(i\cdot):=\{\widehat{f_a}+F_{AP}:
f_a\in L^{1}(0,+\infty), \; F_{AP}(i\cdot)\in APW_\Sigma \},
$$
where $\Sigma$ is as
described in Remark~\ref{examples_of_R_in_AP+}.

An application of our main result (Theorem \ref{thm_main}) yields the
following Nyquist criterion. We note that invertibility of $f$ in
$\calA$ just means that $f \in \calA$, it is nonzero on
$i\mR$ and the almost periodic part of $f$ is bounded away
from zero on $i\mR$ by Lemma~\ref{lemma_inv_A}.

\begin{corollary}
\label{corollary_nyquist_for_calA}
  Let $R$ be a unital full subring of $\calA^+$.  Let $P\in \mS(R, p,
  m)$ and $C \in \mS(R, m, p)$.  Moreover, let $P=N_P D_P^{-1}$ be a
  right coprime factorization of $P$, and $C=\widetilde{D}_C^{-1}
  \widetilde{N}_C $ be a left coprime factorization of $C$.  Then the
  following are equivalent:
\begin{enumerate}
\item $C$ stabilizes $P$.
\item 
\begin{enumerate}
\item $\det (I-CP) \in \calA$,
\item $\det (I-CP)$, $\det D_P$, $\det \widetilde{D}_C$ are all
  nonzero on $i\mR$ and their almost periodic parts are 
  bounded away from zero on $i\mR$, and  
\item $W(\det  (I-CP))+ W(\det D_P)+ W(\det \widetilde{D}_C)=(0,0)$.
\end{enumerate}
\end{enumerate}
\end{corollary}

\begin{remark}
  It was shown in \cite{BruSas} that $\calA^+$ is a projective free
  ring.  Thus the set $\mS(\calA^+, p, m)$ of plants possessing a left
  and a right coprime factorization coincides with the class of plants
  that are stabilizable by \cite[Theorem~6.3]{Qua}.

  Corollary \ref{corollary_nyquist_for_calA} was known in the special
  case when $R=\calA^+$; see \cite{CalDes}.
\end{remark}

\subsection{The complex Borel measure algebra}

Let $\calM$ denote the set of all complex Borel measures on $\mR$.
Then $\calM_+$ is a complex vector space with addition and scalar
multiplication defined as usual, and it becomes a complex algebra if
we take convolution of measures as the operation of multiplication.
With the norm of $\mu$ taken as the total variation of $\mu$,
$\calM$ is a Banach algebra. Recall that the {\em total variation}
$\|\mu\|$ {\em of } $\mu$ is defined by
$$
\|\mu\|=\sup \sum_{n=1}^\infty |\mu (E_n)| ,
$$
the supremum being taken over all {\em partitions of } $\mR$, that is
over all countable collections $(E_n)_{n\in \mN}$ of Borel subsets of
$\mR$ such that $E_n \bigcap E_m =\emptyset$ whenever $m\neq n$ and
$\mR=\bigcup_{n\in \mN} E_n$.  Let $\calM^+$ denote the Banach
subalgebra of $\calM$ consisting of all measures $\mu \in \calM$ whose
support is contained in the half-line $[0,+\infty)$. The following
result was obtained in \cite{Tay}:

\begin{proposition}
\label{prop_Tay}
  If $\mu $ is an invertible measure in $\calM$, then there exist an
  integer $n\in \mZ$, a real number $c\in \mR$ and a measure $\nu \in
  \calM$ such that
  $$
 \mu =\rho^n \ast e^{\nu} \ast \delta_c.
$$
Here $\delta_c$ denotes the Dirac measure supported at $c$.
The measure $\rho$ is given by $d\rho(t)= d\delta_0(t) + 2 {\mathbf
  1}_{[0,\infty)} (t) e^{-t} dt$, where ${\mathbf 1}_{[0,+\infty)} $
is the indicator function of the interval $[0,+\infty)$.
\end{proposition}

\noindent We now define $I: \inv \calM \rightarrow \mR \times \mZ$ as follows:
$$
I (\mu)= (c,n),
$$
where $\mu = \rho^n \ast e^{\nu} \ast \delta_c \in \inv \calM$. It can
be shown that $I$ is well-defined, since in any such decomposition,
the $n$, $\nu$ and $c$ are unique.

\begin{lemma}
Let
\begin{eqnarray*}
  R&:=& \textrm{\em be a unital full subring of }\calM^+,\\
  S&:=& \calM, \\
  G&:=& \mR\times \mZ,\\
  \iota&:=& I.
\end{eqnarray*}
Then {\em (A1)-(A4)} are satisfied.
\end{lemma}
\begin{proof} (A1) and (A2) are clear. (A3) follows from the
  definition of $I$, since $\rho^n \ast \rho^{\widetilde{n}}=
  \rho^{n+\widetilde{n}}$ for all integers $n,m$ and $\delta_c \ast
  \delta_{\widetilde{c}}=\delta_{c+\widetilde{c}}$.

  Finally we check that (A4) holds. Suppose that $\mu \in R\cap (\inv
  \calM)$ is such that $I(\mu)=0$. Then from
  Proposition~\ref{prop_Tay} above, $\mu =\rho^0 \ast e^{\nu} \ast
  \delta_0= e^{\nu}$ for some $\nu \in \calM$. But this implies that
  $\nu$ also has support in $[0,+\infty)$, which can be seen as
  follows. Write $\nu=\nu_1+\nu_2$, where $\nu_1$ has support in
  $[0,+\infty)$ and $\nu_2$ has support in $(-\infty,0]$. It follows
  from $\mu=e^{\nu}$ that $\mu\ast e^{-\nu_1}=e^{\nu_2}$. But $\mu\ast
  e^{-\nu_1}$ has support in $[0,+\infty)$, while $e^{\nu_2}$ has
  support in $(-\infty,0]$. Hence the support of $\nu_2$ must be
  contained in $\{0\}$, and so $\nu$ has support in $[0,+\infty)$. But
  then clearly $e^{-\nu}\in \calM^+$ is an inverse of $\mu$. As $R$ is
  a full subring of $\calM^+$, we conclude that $\mu$ is invertible in
  $R$ as well.

  Conversely, suppose that $\mu\in R\cap (\inv \calM)$ is invertible
  as an element of $R$. Then $\mu$ is also invertible as an element of
  $\calM_+$. Consider the Toeplitz operator $W_{\mu}:L^{2}(0,+\infty)
  \rightarrow L^{2}(0,+\infty) $ given by $W_{\mu} f =P(\mu \ast f)$,
  where $P$ is the canonical projection from $L^{2}(\mR)$ onto
  $L^{2}(0,+\infty)$. Since $\mu$ is in invertible element of
  $\calM^+$, it is immediate that $W_{\mu}$ is invertible. In
  particular, $W_{\mu}$ is Fredholm with Fredholm index $0$. But
  \cite[Theorem~2,~p.139]{DouTay} says that for $\nu \in \inv \calM$,
  $W_{\nu}$ is Fredholm iff $I(\nu)=(0,n)$ for some integer $n$, and
  moreover the Fredholm index of $W_{\nu}$ is then $-n$. Applying this
  result in our case, we obtain that $I(\mu)=(0,0)$. This completes
  the proof.
\end{proof}

An application of our main result (Theorem \ref{thm_main}) yields the
following Nyquist criterion.

\begin{corollary}
Let $R$ be a unital full subring of $\calM^+$.
  Let $P\in \mS(R, p, m)$ and $C \in \mS(R, m, p)$.
  Moreover, let $P=N_P D_P^{-1}$ be a right coprime factorization of
  $P$, and $C=\widetilde{D}_C^{-1} \widetilde{N}_C $ be a left coprime factorization of $C$.
  Then the following are equivalent:
\begin{enumerate}
\item $C$ stabilizes $P$.
\item
\begin{enumerate}
\item $\det (I-CP)$, $\det  D_P$, $\det \widetilde{D}_C$ belong to $\textrm{\em inv }\calM$, and
\item $I(\det  (I-CP))+ I(\det D_P)+ I(\det \widetilde{D}_C)=(0,0)$.
\end{enumerate}
\end{enumerate}
\end{corollary}

\begin{remark}
  It was shown in \cite{BruSas} that $\calM^+$ is a projective free
  ring.  Thus the set $\mS(\calM^+, p, m)$ of plants possessing a left
  and a right coprime factorization coincides with the class of plants
  that are stabilizable by \cite[Theorem~6.3]{Qua}.
\end{remark}

\subsection{The Hardy algebra}

Let $H^\infty(\mD)$ denote the Hardy algebra of all bounded and
holomorphic functions $f:\mD\rightarrow \mC$. Let $H^2(\mD)$ denote
the Hardy Hilbert space. For $f\in L^\infty(\mT)$, we denote by $T_f$
the Toeplitz operator corresponding to $f$, that is, $T_f \varphi =P_+
(M_f \varphi)$, $\varphi \in H^2(\mD)$.  Here $M_f$ denotes the
pointwise multiplication map by $f$, taking $\varphi \in L^2(\mT)$ to
$f\varphi \in L^2(\mT)$, while $P_+: L^2(\mT) \rightarrow H^2(\mD)$ is
the canonical orthogonal projection.

If $f\in \inv (H^\infty(\mD)+C(\mT))$, then $T_f$ is a Fredholm
operator; see \cite[Corollary 7.34]{Dou}. In this case,
let $\ind T_f$ denote the index of the Fredholm operator $T_f$.

Recall the definition of the {\em harmonic extension} of an
$L^\infty(\mT)$-function.

\begin{definition}
If $z=re^{it}$ is in $\mD$ and $f\in L^\infty(\mT)$, then we define
$$
F(z)=\sum_{n=-\infty}^{\infty} a_n r^{|n|}e^{int}
=\frac{1}{2\pi} \int_0^{2\pi} f(e^{i\theta})k_r(t-\theta)d\theta,
$$
where $k_r(\theta)=\displaystyle \frac{1-r^2}{1-2r\cos \theta +r^2}$
and $ a_n=\displaystyle\frac{1}{2\pi} \int_0^{2\pi}
f(e^{i\theta})e^{-2\pi i n \theta}d \theta$.
\end{definition}

We will also use the result given below; see \cite[Theorem~7.36]{Dou}.

\begin{proposition}
\label{prop_Dou}
If $f\in H^\infty(\mD)+C(\mT)$, then $T_f$ is Fredholm iff there exist
$\delta, \epsilon>0$ such that
$$
|F(re^{it})| \geq \epsilon \textrm{ for } 1-\delta <r<1,
$$
where $F$ is the harmonic extension of $f$ to $\mD$. Moreover, in this
case the index of $T_f$ is the negative of the winding number with
respect to the origin of the curve $F(re^{it})$ for $1-\delta <r<1$.
\end{proposition}

\begin{lemma}
 Let
\begin{eqnarray*}
  R&:=& H^\infty(\mD),\\
  S&:=& H^\infty(\mD)+C(\mT), \\
  G&:=& \mZ,\\
  \iota&:=&-\textrm{\em ind }T_{\bullet}.
\end{eqnarray*}
Then {\em (A1)-(A4)} are satisfied.
\end{lemma}
\begin{proof} (A1) and (A2) are clear. (A3) follows from the fact that
  the index of the product of two Fredholm operators is the sum of
  their respective indices; see for example
  \cite[Exercise~2.5.1.(f)]{Nik}.  The `only if' part of (A4) is
  immediate, since if $f$ is invertible as an element of
  $H^\infty(\mD)$, then $T_f$ is invertible, and so $\ind T_f =0$.
  The `if' part of (A4) follows from Proposition~\ref{prop_Dou}.
  Suppose that $f \in H^\infty(\mD)$, that $f$ is invertible as an
  element of $H^\infty(\mD)+C(\mT)$ and that $\ind T_f=0$.  By
  Proposition~\ref{prop_Dou}, it follows that there exist $\delta,
  \epsilon>0$ such that $|F(re^{it})| \geq \epsilon$ for $ 1-\delta
  <r<1$, where $F$ is the harmonic extension of $f$ to $\mD$. But
  since $f \in H^\infty(\mD)$, its harmonic extension $F$ is equal to
  $f$. So $|f(re^{it})| \geq \epsilon$ for $ 1-\delta <r<1$. Also
  since $\iota(f)=0$, the winding number with respect to the origin of
  the curve $f(re^{it})$ for $1-\delta <r<1$ is equal to $0$. By the
  Argument principle, it follows that $f$ cannot have any zeros inside
  $r\mT$ for $1-\delta <r<1$. In light of the above, we can now
  conclude that there is an $\epsilon'>0$ such that $|f(z)|>\epsilon'$
  for all $z\in \mD$. It follows from the corona theorem for
  $H^\infty(\mD)$ that $f$ is invertible as an element of
  $H^\infty(\mD)$.
\end{proof}

An application of Theorem \ref{thm_main} yields the
following Nyquist criterion.

\begin{corollary}
  Let $P\in \mS(H^\infty(\mD), p, m)$ and $C \in \mS(H^\infty(\mD), m,
  p)$.  Moreover, let $P=N_P D_P^{-1}$ be a right coprime
  factorization of $P$, and $C=\widetilde{D}_C^{-1} \widetilde{N}_C $
  be a left coprime factorization of $C$.  Then the following are
  equivalent:
\begin{enumerate}
\item $C$ stabilizes $P$.
\item \begin{enumerate}
\item $\det (I-CP) \in H^\infty (\mD)+C(\mT)$.
\item Let $F_1, F_2, F_3$ be the harmonic extensions to $\mD$, of 
$$
f_1:=\det  (I-CP), \quad f_2:=\det D_P, \quad f_3:=\det \widetilde{D}_C, 
$$
respectively. There exist $\delta, \epsilon>0$ such that
$$
|F_i(re^{it})| \geq \epsilon, \quad 1-\delta <r<1, \quad i=1,2,3.
$$
\item $\iota(\det
  (I-CP))+\iota(\det D_P)+\iota(\det \widetilde{D}_C)=0$.
\end{enumerate}
\end{enumerate}
\end{corollary}

\begin{remark}
  It was proved by Inouye \cite{Ino} that the set $\mS(H^\infty(\mD),
  p, m)$ of plants possessing a left and a right coprime factorization
  coincides with the class of plants that are stabilizable. 
\end{remark}

\subsection{The polydisk algebra}
Let
\begin{eqnarray*}
\mD^n&:=& \{ (z_1, \dots, z_n)\in \mC^n: |z_i| <1 \textrm{ for } i=1,\dots, n\},\\
\overline{\mD^n}&:=& \{ (z_1, \dots, z_n)\in \mC^n: |z_i| \leq 1\textrm{ for } i=1,\dots, n\},\\
\mT^n&:=& \{ (z_1, \dots, z_n)\in \mC^n: |z_i|=1\textrm{ for } i=1,\dots, n\}.
\end{eqnarray*}
The {\em polydisk algebra} $A(\mD^n)$ is the set of all functions $f:
\overline{\mD^n} \rightarrow \mC$ such that $f$ is holomorphic in $\mD^n$
and continuous on $\overline{\mD^n}$.

If $f\in A(\mD^n)$, then the function $f_{d}$ defined by $z\mapsto
f(z,\dots, z): \overline{\mD} \rightarrow \mC$ belongs to the disk
algebra $A(\mD)$, and in particular also to $C(\mT)$. The map
$$
f\mapsto (f|_{\mT^n}, f_d): A(\mD^n) \rightarrow C(\mT^n) \times C(\mT)
$$
is a ring homomorphism. This map is also injective, and this is an
immediate consequence of Cauchy's formula; see \cite[p.4-5]{Rud69}.
We recall the following result; see \cite[Theorem~4.7.2, p.87]{Rud69}.

\begin{proposition}
\label{Rud_prop}
  Suppose that $\Psi=(\psi_1, \dots, \psi_n)$ is a continuous map from
  $\overline{\mD}$ into $\overline{\mD^n}$, which carries $\mT$ into
  $\mT^n$ and the winding number of each $\psi_i$ is positive. Then
  for every $f\in A(\mD^n)$, $f(\Psi(\overline{\mD}) \cup \mT^n)=
  f(\overline{\mD^n})$.
\end{proposition}

\begin{lemma}
  Let
\begin{eqnarray*}
  R&=& \textrm{\em a unital full subring of }A(\mD^n),\\
  S&:=& C(\mT^n) \times C(\mT), \\
  G&:=& \mZ, \\
  \iota&:=&((g, h) \mapsto {\tt w}(h)) .
\end{eqnarray*}
Then {\em (A1)-(A4)} are satisfied.
\end{lemma}
\begin{proof} (A1) and (A2) are clear. (A3) was proved earlier in
  Subsection~\ref{subsection_disk_algebra}. Finally, we will show
  below that (A4) holds, following \cite{DecMurSae}.

  Suppose that $f\in A(\mD^n)$ is such that $f|_{\mT^n} \in \inv
  C(\mT^n)$, $f_d \in \inv C(\mT)$ and that ${\tt w}(f_d)=0$. We use
  Proposition~\ref{Rud_prop}, with $\Psi(z):=(z,\dots, z)$ ($z\in
  \overline{\mD}$). Then we know that $f$ will have no zeros in
  $\overline{\mD^n}$ if $f(\Psi (\overline{\mD}))$ does not contain
  $0$. But since $f_d \in \inv C(\mT)$ and ${\tt w}(f_d)=0$, it
  follows that $f_d$ is invertible as an element of $A(\mD)$ by the
  result in Subsection~\ref{subsection_disk_algebra}. But this implies
  that $f(\Psi (\overline{\mD}))$ does not contain $0$.

  Now suppose that $f\in A(\mD^n)$ with $f|_{\mT^n} \in \inv
  C(\mT^n)$, $f_d \in \inv C(\mT)$, and that it is invertible as an
  element of $A(\mD^n)$. But then in particular, $f_d$ is an
  invertible element of $A(\mD)$, and so again by the result in
  Subsection~\ref{subsection_disk_algebra}, it follows that ${\tt
    w}(f_d)=0$.
\end{proof}

\noindent Besides $A(\mD^n)$ itself, some other examples of such $R$
are:
\begin{enumerate}
\item $\calP$, the set of all polynomials $p:\mC^n\rightarrow \mC$,
\item $RH^\infty(\mD^n)$, the set of all rational functions without
  poles in $\overline{\mD^n}$.
\end{enumerate}

\noindent An application of our main result (Theorem \ref{thm_main})
yields the following Nyquist criterion.

\begin{corollary}
\label{corollary_nyquist_for_polydisk}
  Let $R$ be a unital full subring of $A(\mD^n)$. Let $P\in \mS(R,
  p, m)$ and $C \in \mS(R, m, p)$. Moreover, let $P=N_P D_P^{-1}$ be
  a right coprime factorization of $P$, and $C= \widetilde{D}_C^{-1}  \widetilde{N}_C $ be a
  left coprime factorization of $C$. Then the following are
  equivalent:
\begin{enumerate}
\item $C$ stabilizes $P$.
\item
\begin{enumerate}
\item $\det (I-CP)$, $\det D_P$, $\det \widetilde{D}_C$ belong to
  $\textrm{\em inv }(C(\mT^n)\times C(\mT))$, and
\item $\iota(\det  (I-CP))+\iota(\det D_P)+ \iota(\det \widetilde{D}_C)=0$.
\end{enumerate}
\end{enumerate}
\end{corollary}

\begin{remark}
  By \cite{BruSas}, it follows that $A(\mD^n)$ is a projective free
  ring, since its maximal ideal space the polydisk $\overline{\mD^n}$
  is contractible.  Thus the set $\mS(A(\mD^n), p, m)$ of plants
  possessing a left and a right coprime factorization coincides with
  the class of plants that are stabilizable by
  \cite[Theorem~6.3]{Qua}.

  Corollary \ref{corollary_nyquist_for_polydisk} was known in the special
  case when $R=\calP$; see \cite{DecMurSae}.
\end{remark}

\noindent {\bf Acknowledgement:} I would like to thank Alban Quadrat
for mentioning to me the problem of obtaining a Nyquist criterion for
infinite dimensional control systems, for the references \cite{CalDes}
and \cite{Dav}, and for a discussion on this area.

\end{document}